\documentclass[12pt,reqno]{amsart}

\usepackage{amssymb}
\usepackage{amsmath}
\usepackage{latexsym}
\usepackage{amsthm}

\newtheorem{theorem}{Theorem}[section]

\newtheorem{corollary}{Corollary}[section]
\newtheorem{lemma}{Lemma}[section]
\theoremstyle{definition}
\newtheorem{definition}{Definition}[section]

\theoremstyle{plain}
\theoremstyle{remark}
\newtheorem{remark}{\it Remark}[section]

\renewcommand{\div}{\mathop\mathrm{div}}
\newcommand{\vphi}{\varphi}

\begin{document}
\title
[On the spectrum of the Stokes operator]
 {On the spectrum of the Stokes operator}
\author[A. Ilyin]{Alexei A. Ilyin}
\address
{Keldysh Institute of Applied Mathematics  RAS}
\email{ilyin@keldysh.ru}

\begin{abstract}
We prove Li--Yau-type lower bounds for the eigenvalues
of the Stokes operator and give applications to the
attractors of the Navier--Stokes equations.
\end{abstract}

\thanks
{
This work
was supported in part by the Russian Foundation for
Basic Research,
grant~nos.~06-001-0096 and 05-01-00429,
and by the
the RAS Programme no.1
`Modern problems of theoretical mathematics'
}


\keywords{Stokes operator, Navier--Stokes equations,
attractor dimension,
Lieb--Thirring inequalities.}

\maketitle
\setcounter{equation}{0}
\section{Introduction}\label{S:Intro}

The monotonically ordered eigenvalues
$\{\mu_k\}_{k=1}^\infty$ of the scalar Dirichlet
problem for the Laplacian in
a bounded domain $\Omega\subset\mathbb{R}^n$
$$
-\Delta\vphi_k=\mu_k\vphi_k,
\qquad \vphi_k\vert_{\partial\Omega}=0
$$
satisfy the classical H.Weyl asymptotic formula
$$
\mu_k\sim
\left(
\frac{(2\pi)^n}{\omega_n|\Omega|}
\right)^{2/n}k^{2/n}\quad\text{as $k\to\infty$},
$$
where $|\Omega|$ is the  $n$-dimensional Lebesgue measure
of $\Omega$ and $\omega_n=\pi^{n/2}/\Gamma(1+n/2)$ is the
volume of the unit ball in $\mathbb{R}^n$. This implies that
$$
\sum_{k=1}^m\mu_k\sim\frac n{2+n}
\left(
\frac{(2\pi)^n}{\omega_n|\Omega|}
\right)^{2/n}m^{1+2/n}\quad\text{as $m\to\infty$}.
$$
In fact,
\begin{equation}\label{Li--Yau}
\sum_{k=1}^m\mu_k\ge\frac n{2+n}
\left(
\frac{(2\pi)^n}{\omega_n|\Omega|}
\right)^{2/n}m^{1+2/n}\,.
\end{equation}
This remarkable sharp lower bound was proved in~\cite{Li-Yau}
and holds for  all \ $m=1,2,\dots$ and for any
domain with $|\Omega|<\infty$.

In this paper we prove   Li--Yau-type lower
bounds for the spectrum $\{\lambda_k\}_{k=1}^\infty$
 of the Stokes operator:
 \begin{equation}\label{Stokessmooth}
\aligned
&-\Delta\, v_k \,+\, \nabla\, p_k\,=\,\lambda_kv_k,\\
&\div v_k\,=\,0,\,\,\,v_k\vert_{\partial\Omega }\,=\,0,
\endaligned
\end{equation}
where $\Omega\subset\mathbb{R}^n$, $|\Omega|<\infty$,
$n\ge2$.
The asymptotic behavior of the eigenvalues is known
\cite{Babenko} ($n=3$), \cite{Metiv} ($n\ge2$):
\begin{equation}\label{Stokesassymp}
\lambda_k\sim
\left(
\frac{(2\pi)^n}{\omega_n(n-1)|\Omega|}
\right)^{2/n}k^{2/n}\quad\text{as $k\to\infty$}.
\end{equation}
The main result of this paper proved in
Section~\ref{S:Li--Yau}
 is the following
sharp lower bound for the spectrum of the Stokes operator:
\begin{equation*}
\sum_{k=1}^m\lambda_k\,\ge\,
\frac n{2+n}\left(
\frac{(2\pi)^n}{\omega_n(n-1)|\Omega|}
\right)^{2/n}m^{1+2/n}\,.
\end{equation*}
In addition, $\lambda_1>\mu_1$.
Then in Section~\ref{S:NS} we apply this bound with
$n=2$ and the Lieb--Thirring inequality with improved
constant to the estimates of the dimension of the attractors
of the Navier--Stokes system
with Dirichlet boundary conditions.

\setcounter{equation}{0}
\section{Li--Yau bounds for the spectrum of the Stokes
operator}\label{S:Li--Yau}

Throughout $\Omega$ is an open subset of $\mathbb{R}^n$
with finite $n$-dimensional Lebesgue measure $|\Omega|$:
$$
\Omega\subset\mathbb{R}^n,\ n\ge2,\qquad|\Omega|<\infty.
$$
We recall the basic facts in the theory of the Navier--Stokes
equations \cite{CF88,Lad,TNS,T}.
We denote by $\mathcal{V}$ the set of smooth
divergence-free vector functions with compact supports
$$
\mathcal{V}=\{u:\Omega\to\mathbb{R}^n,
\ u\in \mathbf{C}^\infty_0(\Omega),\
\div u=0\}
$$
and denote by $H$ and $V$ the closure of
$\mathcal{V}$ in  $\mathbf{L}_2(\Omega)$ and
$\mathbf{H}^1(\Omega)$, respectively. The
Helmholtz--Leray orthogonal projection $P$ maps
$\mathbf{L}_2(\Omega)$ onto $H$,
$P:\mathbf{L}_2(\Omega)\to H$.
We have (see~\cite{TNS})
$$
\aligned
&\mathbf{L}_2(\Omega)=H\oplus H^\perp,
\quad H^\perp=\{u\in\mathbf{L}_2(\Omega),
u=\nabla p,\ p\in L_2^{\mathrm{loc}}(\Omega)\},\\
&V\subseteq\{u\in\mathbf{H}^1_0(\Omega),\ \div u=0\},
\endaligned
$$
where  the last inclusion becomes equality
for a bounded $\Omega$ with Lipschitz boundary.

The Stokes operator $A$ is defined by the relation
\begin{equation}\label{Stokes}
(Au,v)=(\nabla u,\nabla v) \quad
\text{for all $u,v$ in $V$}
\end{equation}
and is an isomorphism between $V$ and $V'$.
For a sufficiently smooth $u$
$$
Au=-P\Delta u.
$$
The Stokes operator $A$ is
an unbounded self-adjoint positive operator in $H$
with compact inverse. It has a complete in $H$ and $V$
system of orthonormal eigenfunctions
$\{v_k\}_{k=1}^\infty\in V$
with corresponding eigenvalues
$\{\lambda_k\}_{k=1}^\infty$, $\lambda_k\to\infty$ as
$k\to\infty$:
\begin{equation}\label{spectrum}
Av_k=\lambda_kv_k, \quad
0<\lambda_1\le\lambda_2\le\dots\,.
\end{equation}
Taking the scalar product with $v_k$ we have by
orthonormality and~(\ref{Stokes}) that
\begin{equation}\label{lambdak}
    \lambda_k=\|\nabla v_k\|^2.
\end{equation}
In case when $\Omega$
is a bounded domain with smooth boundary the eigenvalue
problem~(\ref{spectrum})
goes over  to~(\ref{Stokessmooth}).

Our main goal is to prove uniform estimates for
the Fourier transforms of
orthonormal families of divergence-free vector functions
(see Lemma~\ref{l:n-1}).

Given a function $\vphi\in L_2(\Omega)$ we denote
by $\widehat\vphi(\xi)$ the Fourier transform
of its extension by zero outside $\Omega$:
$$
(\mathcal{F}\vphi)(\xi)=
\widehat\vphi(\xi)=\int e^{-i\xi x}\vphi(x)\,dx.
$$
\begin{lemma}\label{l:scal}
Let the family $\{\vphi_k\}_{k=1}^{m}$
be orthonormal in $L_2$:
$
(\vphi_k,\vphi_l)=\delta_{kl}.
$
Then
\begin{equation}\label{sum_scal}
\sum_{k=1}^m|\widehat\vphi_k(\xi)|^2\le|\Omega|.
\end{equation}
\end{lemma}
\begin{proof}
Denoting by $*$ the complex conjugate we have
by orthonormality
$$
\aligned
0\le\int\biggl(e^{-i\xi x}-
\sum_{k=1}^m\widehat{\vphi_k}(\xi)\vphi_k(x)\biggr)
\biggl(e^{-i\xi x}-
\sum_{l=1}^m\widehat{\vphi_l}(\xi)\vphi_l(x)\biggr)^*dx
=|\Omega|-\sum_{k=1}^m|\widehat{\vphi_k}(\xi)|^2.
\endaligned
$$
\end{proof}

\begin{remark}\label{Rem:interpretation}
{\rm
Inequality~\ref{sum_scal} is nothing other than
Bessel's inequality applied to the function
$h(x)=e^{-i\xi x}\vert_{x\in\Omega}$
with $\|h\|_{L_2}^2=|\Omega|$ and the orthonormal
family $\{\vphi_i(x)\}_{i=1}^{m}$ \cite{Li-Yau}.
}
\end{remark}

Next we observe that Lemma~\ref{l:scal} still holds
if we replace the orthonormality condition
by suborthonormality.

\begin{definition}\label{D:suborth}
A family $\{\vphi_i\}_{i=1}^m$ is called suborthonormal
if for any  $\zeta\in\mathbb{C}^m$
\begin{equation}\label{suborth}
\sum_{i,j=1}^m\zeta_i\zeta_j^*(\vphi_i,\vphi_j)\le
\sum_{j=1}^m|\zeta_j|^2.
\end{equation}
\end{definition}

\begin{remark}\label{Rem:suborthonormality}
{\rm
This convenient and flexible notion of suborthonormality was
introduced in~\cite{G-M-T} with real~$\zeta\in\mathbb{R}^m$
and is equivalent to the formally more general
Definition~\ref{D:suborth}.
}
\end{remark}
\begin{lemma}\label{l:scal_suborth}
Let the family $\{\vphi_k\}_{k=1}^{m}$
be suborthonormal.
Then
\begin{equation}\label{sum_scal_suborth}
\sum_{k=1}^m|\widehat{\vphi_k}(\xi)|^2\le|\Omega|.
\end{equation}
\end{lemma}
\begin{proof} As in Lemma~\ref{D:suborth}
with~(\ref{suborth}) instead of orthonormality we
have
$$
\aligned
0&\le\int\biggl(e^{-i\xi x}-
\sum_{k=1}^m\widehat{\vphi_k}(\xi)\vphi_k(x)\biggr)
\biggl(e^{-i\xi x}-
\sum_{l=1}^m\widehat{\vphi_l}(\xi)\vphi_l(x)\biggr)^*dx
=\\&=
|\Omega|-2\sum_{k=1}^m|\widehat{\vphi_k}(\xi)|^2+
\sum_{k,l=1}^m\widehat {\vphi_k}(\xi)\widehat {\vphi_l}(\xi)^*
(\vphi_k,\vphi_l)\le
|\Omega|-\sum_{k=1}^m|\widehat{\vphi_k}(\xi)|^2.
\endaligned
$$
\end{proof}

We now turn to orthonormal families of vector functions
$\{u_k\}_{k=1}^m$, $u_k=(u_k^1,\dots
u_k^n)$.
\begin{lemma}\label{l:vec_suborth}
Let the family of vector functions
$\{u_k\}_{k=1}^m$ be orthonormal
 in $\mathbf{L}_2(\Omega)$
and let $Q$ be an arbitrary orthogonal projection.
Then the family
$\{Qu_k\}_{k=1}^m$
 is suborthonormal.
\end{lemma}
\begin{proof}
 We set $u_k=v_k+w_k$, $v_k=Qu_k$
and $w_k=(I-Q)u_k$.
Then $(v_k,w_l)=0$ for all $k,l=1,\dots,n$ and
$(u_k,u_l)=(v_k,v_l)+(w_k,w_l)$. Therefore
$$
\aligned
\sum_{k,l=1}^m\zeta_k\zeta_l^*(v_k,v_l)&=
\sum_{k,l=1}^m\zeta_k\zeta_l^*(u_k,u_l)
-
\sum_{k,l=1}^m\zeta_k\zeta_l^*(w_k,w_l)=\\&=
\sum_{k=1}^m|\zeta_k|^2-\bigl\|\sum\nolimits_{k=1}^m\zeta_kw_k\bigr\|^2
\le
\sum_{k=1}^m|\zeta_k|^2.
\endaligned
$$
\end{proof}
\begin{corollary} If  the family of vector functions
$\{u_k\}_{k=1}^m$ is orthonormal
in $\mathbf{L}_2$, then
\begin{equation}\label{n}
\sum_{k=1}^m|\widehat {u_k}(\xi)|^2\le n|\Omega|.
\end{equation}
\end{corollary}
\begin{proof}
By Lemma~\ref{l:vec_suborth}   each
 family
$\{u_k^j\}_{k=1}^m$ is suborthonornal
 $j=1,\dots,n$,
and~(\ref{n}) follows from Lemma~\ref{l:scal_suborth}.
\end{proof}
The next lemma is the central point  in the proof
of the lower bounds for the spectrum
 and says that
under the divergence-free condition the
estimate~(\ref{n}) goes over to~(\ref{n-1}).
\begin{lemma}\label{l:n-1}
If the family of vector functions $\{u_k\}_{k=1}^m$ is
orthonormal and $u_k\in \mathbf{H}^1_0(\Omega)$,
$\div u_k=0$, $k=1,\dots,m$, then
\begin{equation}\label{n-1}
\sum_{k=1}^m|\widehat {u_k}(\xi)|^2\le (n-1)|\Omega|.
\end{equation}
\end{lemma}
\begin{proof}
We first observe that for all $\xi\in \mathbb{R}^n_\xi$
$$
\xi\cdot \widehat {u_k}(\xi)=
\xi\cdot\int e^{-i\xi x}\,u_k(x)\,dx=
i\int u_k\cdot\nabla_xe^{-i\xi x}\,dx
=-i\int e^{-i\xi x}\div u_k\,dx=0.
$$
Let $\xi_0\ne0$ be of the form:
\begin{equation}\label{xi0}
\xi_0=(a,0,\dots,0), \qquad a\ne0.
\end{equation}
Since $\xi_0\cdot \widehat {u_k}(\xi_0)=0$, it follows that
$\widehat u_k^1(\xi_0)=0$ for  $k=1,\dots,m$, which
in view of Lemmas~\ref{l:vec_suborth} and \ref{l:scal_suborth}
proves the estimate~(\ref{n-1})
for $\xi$ of the form~(\ref{xi0}):
$$
\sum_{k=1}^m|\widehat {u_k}(\xi_0)|^2=
\sum_{j=2}^n\sum_{k=1}^m|\widehat {u}_k^j(\xi_0)|^2
\le (n-1)|\Omega|.
$$

The general case reduces to the case~(\ref{xi0}) by the
corresponding rotation. Let $\rho$ be a rotation of
$\mathbb{R}^n$
about the origin represented by the
orthogonal $(n\times n)-$matrix $\rho$ with
entries $\rho_{ij}$.
Given a vector function $u(x)=(u^1(x),\dots,u^n(x))$ we
consider the vector function
$$
u_\rho(x):=\rho\, u(\rho^{-1}x),
\qquad x\in\rho\Omega.
$$
Let us calculate the divergence of $u_\rho(x)$.
Setting $\rho^{-1}x=y$,
$y_l=\sum_{k}(\rho^{-1})_{lk}\,x_k$
 we have
$$
\frac{\partial u_\rho^i(x)}{\partial x_i}=
\frac{\partial }{\partial x_i}
\biggl(
\sum_{j}\rho_{ij}u^j(y)
\biggr)=
\sum_j\rho_{ij}\sum_l\frac{\partial u^j(y)}{\partial y_l}
\frac{\partial y_l}{\partial x_i}=
\sum_j\rho_{ij}\sum_l\frac{\partial u^j(y)}{\partial y_l}
(\rho^{-1})_{li}.
$$
Therefore
$$
\div u_\rho(x)=\sum_{i,j,l}
\rho_{ij}\frac{\partial u^j(y)}{\partial y_l}
(\rho^{-1})_{li}=
\sum_{j,l}\frac{\partial u^j(y)}{\partial y_l}\sum_i
(\rho^{-1})_{li}\rho_{ij}=\div u(y).
$$
In addition,
$$
(u_\rho,v_\rho)=
\int\rho u(\rho^{-1}x)\cdot\rho v(\rho^{-1}x)\,dx=
\int u(\rho^{-1}x)\cdot v(\rho^{-1}x)\,dx=
\int u(y)\cdot v(y)\,dy=(u,v).
$$
Combining this we obtain that the family
$\{(u_k)_\rho\}_{k=1}^m$ belongs to
$\mathbf{H}^1_0(\rho\Omega)$,
is orthonormal and $\div (u_k)_\rho=0$.

Next we calculate $\widehat{u_\rho}$
and show that
\begin{equation}\label{commut}
\widehat{u_\rho}(\xi)=
\rho\widehat{u}(\rho^{-1}\xi).
\end{equation}
In fact,
$$
\aligned
(\mathcal{F}u_\rho)(\xi)=
\widehat{u_\rho}(\xi)=
\int e^{i\xi\cdot x}u_\rho(x)\,dx=
\rho\int e^{i\xi\cdot x}u(\rho^{-1}x)\,dx=\\
\rho\int e^{i\xi\cdot\rho y}u(y)\,dy=
\rho\int e^{i\rho^{-1}\xi\cdot y}u(y)\,dy=
\rho\widehat{u}(\rho^{-1}\xi).
\endaligned
$$

We now fix an arbitrary
$\xi\in\mathbb{R}^n$, $\xi\ne0$ and
set
$\xi_0=(|\xi|,0,\dots,0)$. Let $\rho$ be the
 rotation such that $\xi=\rho^{-1}\xi_0$.
Then we have
$$
\sum_{k=1}^m|\widehat {u_k}(\xi)|^2=
\sum_{k=1}^m|\widehat {u_k}(\rho^{-1}\xi_0)|^2=
\sum_{k=1}^m|\rho^{-1}\widehat {(u_k)_\rho}(\xi_0)|^2=
\sum_{k=1}^m|\widehat {(u_k)_\rho}(\xi_0)|^2
\le (n-1)|\Omega|,
$$
where we have used~(\ref{commut}) and the fact that
inequality~(\ref{n-1}) has been proved for $\xi$ of
the form~(\ref{xi0}) for any orthonormal family
of divergence-free vector functions.
Finally, the estimate (\ref{n-1}) is extended to $\xi=0$ by
continuity (observe that $u_k\in L_1$ since
$|\Omega|<\infty$ and hence the
Fourier transforms $\widehat{u_k}$
are continuous.)
\end{proof}
\begin{remark}\label{Rem:inH}
{\rm
In fact, (\ref{n-1}) holds
under milder assumption that
$u_k\in H$, $k=1,\dots,m$.
}
\end{remark}

We need the following lemma from~\cite{Li-Yau},
whose proof we give for the sake of completeness.
\begin{lemma}\label{l:ineq}{\rm (See \cite{Li-Yau}.)}
 Let a function $f(\xi )$, $f:\mathbb{R}^n\to\mathbb{R}$
  satisfy
$$
0\le f(\xi)\le M_1 \qquad\text{and}\qquad
\int|\xi|^2f(\xi)d\xi\le M_2.
$$
  Then
\begin{equation}\label{multineq}
\int f(\xi )\,d\xi\,\le\,(M_1\omega_n)^{ 2/(2+n)}
(M_2(2+n)/n )^{ n/(2+n)}.
\end{equation}
\end{lemma}
\begin{proof}
We first observe that (\ref{multineq}) turns
into equality for a constant multiple of the
characteristic function $g(\xi)$
of any ball centered at the origin in $\mathbb{R}^n$.
We set
$$
g(\xi)=\left\{
\begin{array}{ll}
    M_1    , & |\xi|\le R, \\
    0, & |\xi|>R.\\
\end{array}
\right.
$$
Then
$(|\xi|^2-R^2)(f(\xi)-g(\xi))\ge 0$ so that
$$
R^2\int(f(\xi)-g(\xi)d\xi\le
\int|\xi|^2(f(\xi)-g(\xi)d\xi\le0,
$$
where the second inequality holds
provided that $R$ is defined by the equality
$$
\int|\xi|^2 g(\xi)d\xi=M_2.
$$
Hence
$$
\int f(\xi )\,d\xi\,\le
\int g(\xi )\,d\xi\,=\,(M_1\omega_n)^{ 2/(2+n)}
(M_2(2+n)/n )^{ n/(2+n)}.
$$
\end{proof}

We can now formulate our main results.
\begin{theorem}\label{T:main}
Suppose that the family of vector functions
$\{u_k\}_{k=1}^m\in \mathbf{H}^1_0(\Omega)$ is orthonormal
and, in addition, $\div u_k=0$, $k=1,\dots,m$. Then
\begin{equation}\label{lowerorth}
\sum_{k=1}^m\|\nabla u_k\|^2\,\ge\,
\frac n{2+n}\left(
\frac{(2\pi)^n}{\omega_n(n-1)|\Omega|}
\right)^{2/n}m^{1+2/n}\,.
\end{equation}
\end{theorem}
\begin{proof} We set
$$
f(\xi)=\sum_{k=1}^m|\widehat{u_k}(\xi)|^2.
$$
By Lemma~\ref{l:n-1} and the Plancherel theorem
$f$ satisfies\begin{enumerate}
    \item  $0\le f(\xi)\le(n-1)|\Omega|$;
    \item  $\int f(\xi)\,d\xi=(2\pi)^nm$;
    \item  $\int |\xi|^2f(\xi)\,d\xi=(2\pi)^n
            \sum_{k=1}^m\|\nabla u_k\|^2$.
\end{enumerate}
Using Lemma~\ref{l:ineq} we find that
$$
(2\pi)^nm=\int f(\xi)\,d\xi\le
\biggl((n-1)|\Omega|\,\omega_n\biggr)^{2/(2+n)}
\biggl((2\pi)^n\sum_{k=1}^m\|\nabla u_k\|^2
(2+n)/n\biggr)^{n/(2+n)},
$$
which is~(\ref{lowerorth}).
\end{proof}

\begin{theorem}\label{Th:main}
The eigenvalues $\lambda_k$ of the Stokes
operator satisfy the following lower bound:
\begin{equation}\label{lowerStokes}
\sum_{k=1}^m\lambda_k\,\ge\,
\frac n{2+n}\left(
\frac{(2\pi)^n}{\omega_n(n-1)|\Omega|}
\right)^{2/n}m^{1+2/n}\,.
\end{equation}
\end{theorem}
\begin{proof}
Since $ V\subseteq\{u\in\mathbf{H}^1_0(\Omega),\
\div u=0\}$ we can chose  the first $m$ eigenvectors
 for the $u_k$'s in~(\ref{lowerorth})
and taking into account~(\ref{lambdak})
we obtain~(\ref{lowerStokes}).
\end{proof}
\begin{remark}\label{Rem:sharp}
{\rm
In view of the asymptotics~(\ref{Stokesassymp})
this lower bound is sharp in the sense that
the inequality with
the coefficient of $m^{1+2/n}$  larger
than in~(\ref{lowerStokes})
cannot hold for a sufficiently large $m$.
}
\end{remark}
\begin{remark}\label{Rem:nonsharp}
{\rm
Weaker lower bounds based on the estimate~(\ref{n})
$$
\sum_{k=1}^m\lambda_k\,\ge\,
\frac n{2+n}\left(
\frac{(2\pi)^n}{\omega_nn|\Omega|}
\right)^{2/n}m^{1+2/n}\,
$$
have earlier been proved in~\cite{I96} for $n=2,3$.
}
\end{remark}
\begin{remark}\label{Rem:maxmin}
{\rm
In fact, for any orthonormal family
$\{u_k\}_{k=1}^m\in V$
we have
$$
\sum_{k=1}^m\|\nabla u_k\|^2\,\ge\,
\sum_{k=1}^m\lambda_k.
$$
}
\end{remark}

\begin{corollary}\label{Cor:individual}
Each eigenvalue $\lambda_k$ satisfies
\begin{equation}\label{each}
   \lambda_k\,\ge\,
   \frac n{2+n}\left(
\frac{(2\pi)^n}{\omega_n(n-1)|\Omega|}
\right)^{2/n}k^{2/n}\,,
\end{equation}
while $\lambda_1$ satisfies
\begin{equation}\label{first}
\lambda_1>\mu_1\ge\frac n{2+n}
\left(
\frac{(2\pi)^n}{\omega_n|\Omega|}
\right)^{2/n}\,.
\end{equation}
\end{corollary}
\begin{proof}
The sequence $\{\lambda_k\}_{k=1}^\infty$
is nondecreasing and (\ref{each}) is obvious.
Since $V\subset \mathbf{H}^1_0(\Omega)$,
$$
\mu_1=\min_{u\in\mathbf{H}^1_0(\Omega)}
\frac{\|\nabla u\|^2}{\|u\|^2}\le
\min_{u\in V}
\frac{\|\nabla u\|^2}{\|u\|^2}=\lambda_1
$$
and the second inequality in~(\ref{first})
is~(\ref{Li--Yau}) with $m=1$.
Let us prove that $\lambda_1>\mu_1$. Suppose that
$
\mu_1={\|\nabla u_0\|^2}/{\|u_0\|^2}$
for some   $\ u_0\in \mathbf{H}^1_0(\Omega)$.
It is well known that $\mu_1$ is a simple eigenvalue
with unique (up to a constant factor)
eigenfunction $\vphi_1$. Therefore
any such $u_0$ is of the form
$
u_0(x)=(l_1\vphi_1(x),l_2\vphi_1(x),\dots,l_n\vphi_1(x))
$
for some constants $l_1,\dots,l_n$, $|l|>0$.
(Without loss of generality we can assume that
$|l|=1$.)
Now $\lambda_1=\mu_1$ if and only if
$u_0$ so obtained satisfies, in addition, $\div u_0=0$.
Therefore
$\frac{\partial\vphi_1}{\partial l}=\div u_0=0$,
and $\vphi_1$ is constant along the lines parallel to $l$,
which is impossible.
\end{proof}

\setcounter{equation}{0}
\section{Applications to the
Navier--Stokes system}\label{S:NS}

We write the two-dimensional Navier--Stokes
system as an evolution equation in $H$
\begin{equation}\label{NSE}
\partial_t u\,+\,\nu A\,u\,+\,B(u,u)\,=\,f,\quad
u(0)=u_0,
\end{equation}
where $A=-P\Delta$ is the Stokes operator and
$B(u,v)=P(\sum_{i=1}^2u^i\partial_iv)$.
The equation (\ref{NSE}) generates the semigroup
$S_t:H\to H$, $S_tu_0=u(t)$, which
 has a compact global attractor
$\mathcal{A}\Subset H$ (see, for instance,
\cite{B-V},\cite{CF88},\cite{FMRT},\cite{T} for the case
of a domain with
 smooth boundary $\partial\Omega$, and
\cite{LadNon},\cite{Rosa} for a nonsmooth domain).
The attractor $\mathcal{A}$ is the maximal
strictly invariant compact  set.

\begin{theorem}\label{T:dim}
The fractal dimension of  $\mathcal{A}$ satisfies the
following estimate
\begin{equation}\label{dim}
\dim_F\mathcal{A}\le
\frac1{(8\sqrt{3}\,\pi)^{1/2}}\,(\lambda_1|\Omega|)^{1/2}
\frac{\|f\|}{\lambda_1\nu^2}\,<\,
\frac1{4\pi 3^{1/4}}\,\frac{\|f\||\Omega|}{\nu^2}\,.
\end{equation}
\end{theorem}
\begin{proof}
Since for the proof of~(\ref{dim}) we  need to
use in~\cite[Theorem 4.1]{Ch-I} the new improved constants
in the Lieb--Thirring inequality~(\ref{L-T}) below
and in the lower bound~(\ref{lowerorth}) for $n=2$,
the proof of the theorem  will only be outlined.
The solution semigroup $S_t$ is uniformly
 differentiable in $H$
with differential $L(t,u_0):\xi\to U(t)\in H$,
where $U(t)$ is the solution of the variational equation
\begin{equation}\label{var-eq}
\partial_tU=-\nu AU-
B(U,u(t))-B(u(t),U)=:{\mathcal L}(t,u_0)U, \qquad U(0)=\xi.
\end{equation}

We estimate
the numbers $q(m)$ (the sums of the first $m$
global Lyapunov exponents):
\begin{equation}\label{def-q(m)}
q(m)\le\limsup_{t\to\infty}\ \sup_{u_0\in {\mathcal A}}\ \
\sup_{\{v_j\}_{j=1}^m\in V}
\frac{1}t
\int_0^t \sum_{j=1}^m\bigl({\mathcal L}(\tau,u_0)v_j,v_j\bigr)d\tau,
\end{equation}
where $\{v_j\}_{j=1}^m\in V$ is an arbitrary
orthonormal system of dimension $m$
\cite{B-V},\cite{C-F85},\cite{CF88},\cite{T}.
$$
\aligned
&\sum_{j=1}^m({\mathcal L} (t,u_0)v_j,v_j)\,=\,
-\nu\sum_{j=1}^m\|\nabla v_j\|^2\,-\,\int\sum_{j=1}^m
\sum_{k,i=1}^2
v_j^k\partial_ku^iv_j^idx\,\le\\
-&\nu\sum_{j=1}^m\|\nabla v_j\|^2\,+\, 2^{-1/2}
\int\rho (x)|\nabla u(t,x)|\,dx\,\le\\
-&\nu\sum_{j=1}^m\|\nabla v_j\|^2\,+\,
2^{-1/2}\|\rho\|\|\nabla u\|\,\le\\\,
-&\nu\sum_{j=1}^m\|\nabla v_j\|^2\,+\,
2^{-1/2}
\biggl(c_{\mathrm{LT}}
\sum_{j=1}^m\|\nabla v_j\|^2\biggr)^{1/2}
\|\nabla u(t)\|\,\le\\
-&\frac\nu2\sum_{j=1}^m\|\nabla v_j\|^2\,+\,
\frac{c_\mathrm{LT}}{4\nu}\|\nabla u(t)\|^2\,\le\,
-\frac{\nu c_\mathrm{sp}m^2}{2|\Omega|}\,+\,
\frac{c_\mathrm{LT}}{4\nu}\|\nabla u(t)\|^2\,,
\endaligned
$$
Here we used the inequality
$|\sum_{k,i=1}^2
v^k\partial_ku^iv^i|=
|\nabla u\,v\cdot v|\le
2^{-1/2}|\nabla u||v|^2$~\cite[Lemma~4.1]{Ch-I},
then~(\ref{L-T}), and, finally,
(\ref{lowerorth}), written for $n=2$
and the orthonormal family $\{v_j\}_{j=1}^m\in V$ as follows
\begin{equation}\label{n=2}
\sum_{k=1}^m\|\nabla v_k\|^2\,\ge\,
\frac {c_{\mathrm{sp}}m^2}{|\Omega|}
\,,\qquad c_{\mathrm{sp}}=2\pi.
\end{equation}

Using the well-known estimate
$$
\limsup_{t\to\infty}\sup_{u_0\in\mathcal{A}}
\frac 1t\int_0^t\|\nabla u(\tau)\|^2d\tau\,
\le\,\frac{\|f\|^2}{\lambda_1\nu^2}\,=\,
\lambda_1\nu^2G^2,\qquad
G=\frac{\|f\|}{\lambda_1\nu^2}
$$
for the solutions lying on the attractor we
obtain  for the
numbers $q(m)$:
$$
q(m)\le
-\frac{\nu c_{\mathrm{sp}}m^2}{2|\Omega|}+
\frac{\nu\lambda_1c_{\mathrm{LT}}G^2}4\,.
$$
It was shown in~\cite{C-F85} (see also~\cite{CF88},\cite{T})
and in~\cite{Ch-I}, respectively,
that both the Hausdorff and fractal dimensions of $\mathcal{A}$
are bounded by the number $m_*$ for which
$q(m_*)=0$.
This gives that
$$
\dim_F\mathcal{A}\le
\biggl(\frac{c_{\mathrm{LT}}}{2c_{\mathrm{sp}}}\biggr)^{1/2}
(\lambda_1|\Omega|)^{1/2}G,
$$
which in view of~(\ref{L-T}) and~(\ref{n=2})
 proves the first inequality in~(\ref{dim}),
while the second inequality follows from
(\ref{first}) with $n=2$: $\lambda_1>2\pi/|\Omega|$.

\end{proof}

\begin{theorem}\label{T:L-T}
Let the family
$\{v_j\}_{j=1}^m\in \mathbf{H}_0^1(\Omega)$,
$\Omega\subseteq\mathbb{R}^2$ be orthonormal  and
$\div v_j=0$, $j=1,\dots,m$.
 Then the following inequality holds for
 $\rho(x)=\sum_{k=1}^m|v_k(x)|^2$:
\begin{equation}\label{L-T}
\|\rho\|^2=
\int\biggl(\sum_{j=1}^m|v_j(x)|^2\biggr)^2dx \,\le\,
c_\mathrm{LT}
\sum_{j=1}^m \| \nabla v_j\|^2,\qquad
c_\mathrm{LT}
\le\frac1{2\sqrt{3}}\,.
\end{equation}
\end{theorem}
\begin{proof}
It was proved in~\cite{Ch-I}, \cite{I-MS2005} that
the best (by notational definition) constant
$c_\mathrm{LT}$
in~(\ref{L-T}) satisfies
$$
c_\mathrm{LT}
\le 4\mathrm{L}_{1,2},
$$
where the constant $\mathrm{L}_{1,2}$ comes from the
Lieb--Thirring spectral estimate \cite{L--T}
$$
\sum_{\mu_j<0}|\mu_j|^\gamma\le \mathrm{L}_{\gamma,n}
\int_{\mathbb{R}^n} f(x)^{\gamma+n/2}dx
$$
for the negative eigenvalues of the scalar
Schr\"odinger operator $-\Delta-f$
in $\mathbb{R}^n$, $f\ge0$.
For $\mathrm{L}_{\gamma,n}$ we always have
$$
\mathrm{L}_{\gamma,n}\ge
\mathrm{L}_{\gamma,n}^{\mathrm{cl}}:=
\frac{\Gamma(\gamma+1)}{(4\pi)^{n/2}\Gamma(\gamma+n/2+1)}\,.
$$

It was recently shown in~\cite{D-L-L} that for $n\ge1$
$$
\mathrm{L}_{\gamma,n}\le R\cdot
\mathrm{L}_{\gamma,n}^{\mathrm{cl}},
\qquad R=\pi/\sqrt{3}=1.8138\dots\,,
\qquad\gamma\ge1,
$$
which improves the previous important estimate
$\mathrm{L}_{\gamma,n}\le 2
\mathrm{L}_{\gamma,n}^{\mathrm{cl}}$
established in~\cite{H-L-W}.
Hence
$c_\mathrm{LT}\le
4\, R\, \mathrm{L}_{1,n}^{\mathrm{cl}}
=1/(2\sqrt{3})$.
The proof is complete.
\end{proof}

\begin{remark}\label{Rem:L-T}
{\rm The idea to use Lieb--Thirring inequalities
in the context of the Navier--Stokes equations~\cite{Lieb}
has led to estimates of dimension that are linear with
respect to the Grashof number $G$~\cite{T85}. First explicit
estimates for the dimension of the attractors were
obtained in~\cite{I96} and improved in~\cite{Ch-I}. The
explicit constants in~(\ref{dim}) are further improvements
(by the factor $(2\cdot(2/R))^{1/2}=1.485\dots$) of the
corresponding constants in~\cite{Ch-I}.
}
\end{remark}

\section*{Acknowledgments}
The author acknowledges helpful
discussions with V.V.\,Chepyzhov and Yu.G.\,Rykov.

\end{document}